\documentclass[twoside,12pt]{amsart}
\usepackage{amsmath,amssymb,amsthm,longtable,graphicx}
\usepackage{txfonts}

\address{Graduate School of Mathematics Nagoya University, Chikusa-ku
Nagoya 464-8602 Japan}
\email{kazuto.iijima@math.nagoya-u.ac.jp}
\thanks{}
\dedicatory{}

\newtheorem{definition}{Definition.}[section]
\newtheorem{theorem}[definition]{Theorem.}
\newtheorem{proposition}[definition]{Proposition.}
\newtheorem{example}[definition]{Example.}
\newtheorem{corollary}[definition]{Corollary.}
\newtheorem{lemma}[definition]{Lemma.}

\newtheorem{problem}[definition]{Problem.}

\begin{document}

\title[]{A $q$-multinomial expansion of LLT coefficients and plethysm multiplicities}
\author[K.~IIJIMA]{Kazuto Iijima}
\date{}
\maketitle

\thispagestyle{empty}

\begin{abstract}
Lascoux, Leclerc and Thibon\cite{LLT} introduced a family of symmetric polynomials, called LLT polynomials.
We prove a $q$-multinomial expansion of the coefficients of LLT polynomials 
in the case where $ \boldsymbol{\mu} = \underbrace{(\mu,\cdots,\mu)}_{n}$
and define a $q$-analog of a sum of the plethysm multiplicities. 

\end{abstract}

%
%

\section{Introduction}

Lascoux, Leclerc and Thibon introduced a family of symmetric polynomials, called LLT polynomials, by using combinatorial objects,
  \textit{ribbon tableaux} and their \textit{spins}\cite{LLT}.
They showed that LLT polynomials are related to the Fock space representation of the quantum affine algebra $U_q(\widehat{sl_n})$.
Leclerc and Thibon gave a $q$-analog of Littlewood-Richardson coefficients by using LLT polynomials\cite{LT}.
Haglund, Haiman, Loehr, Remmel, and Ulyanov\cite{HH1} gave a variant definition of LLT polynomials as follows.

Let $$ \boldsymbol{\mu} = (\mu^{(0)},\mu^{(1)},\cdots,\mu^{(n-1)}) $$ 
be an $n$-tuple of Young diagrams.
Let $\mathrm{SSTab}(\boldsymbol{\mu})$ be the set of tuples of semistandard tableaux 
  $\boldsymbol{T}=(T_0,\cdots,T_{n-1})$, where $T_i$ has shape $\mu^{(i)}$.
Given $\boldsymbol{T}=(T_0,T_1,\cdots,T_{n-1}) \in \mathrm{SSTab}(\boldsymbol{\mu})$,
  we say that two entries $T^{(i)}(u) > T^{(j)}(v)$ form an \textit{Inversion} if either
\begin{align*}
(\mathrm{i} ). \,\, i<j  \text{ and } c(u)=c(v), \text{ or} \\
(\mathrm{ii}). \,\, i>j  \text{ and } c(u)=c(v)-1 ,
\end{align*}
where $c(u)$ is the content of $u$.
Denote by Inv$(\boldsymbol{T})$ the number of Inversions in $\boldsymbol{T}$.

Then the \textit{LLT polynomial} indexed by $\boldsymbol{\mu}$ is defined as 
$$ G_{\boldsymbol{\mu}}(x;q) =
      \sum_{\boldsymbol{T} \in \mathrm{SSTab}(\boldsymbol{\mu})} q^{\mathrm{Inv}(\boldsymbol{T})} x^{\boldsymbol{T}} .$$
Also, we denote $G_{\boldsymbol{\mu},\nu}(q)$ by the coefficient of
$x^{\nu}=x_1^{\nu_1}x_2^{\nu_2}\cdots$ in the LLT polynomial $G_{\boldsymbol{\mu}}(x;q)$, and call it \textit{LLT coefficient}.

In this paper, we are interested in the case where 
$\mu^{(0)}=\mu^{(1)} = \cdots = \mu^{(n-1)} = \mu $, 
i.e. $$ \boldsymbol{\mu} = \underbrace{(\mu,\cdots,\mu)}_{n}.$$

\vspace{1em}

Our main results are the following two expressions of the LLT coefficient $G_{\boldsymbol{\mu},\nu}(q)$;
 Theorem.\ref{fermionicformula} and Theorem.\ref{mainthm}.

$\bf{Theorem.A}\,(\bf{Th.\ref{fermionicformula}}.)$
($q$-multinomial expansion of LLT coefficients)

\it{Let} $\boldsymbol{\mu}=(\underbrace{\mu,\ldots,\mu}_n)$.
Then
$$ G_{\boldsymbol{\mu},\nu}(q)=\sum_{(T_0 \leq \cdots \leq T_{n-1}) \in \mathrm{SSTab}(\boldsymbol{\mu},\nu)}
    q^{\mathrm{Inv}(T_0 \leq \cdots \leq T_{n-1})} \biggl[ \begin{array}{c} n \\ \rho \end{array} \biggr]_q,  $$
where the summation is taken over all increasing $n$-tuples of semistandard tableaux with repect to the total order 
defined in \ref{totalorder} and $\rho$ is the composition which determined by 
$T_0=\cdots=T_{\rho_1-1} < T_{\rho_1}=\cdots=T_{\rho_1+\rho_2-1} < T_{\rho_1+\rho_2}= \cdots$.

\vspace{1em}

$\bf{Theorem.B}\,(\bf{Th.\ref{mainthm}}).$
\it{Let} $ \boldsymbol{\mu} = \underbrace{(\mu,\cdots,\mu)}_{n} .$
There is a statistics $\alpha \colon \mathrm{SSTab}(\boldsymbol{\mu})
     \rightarrow \mathbb{Z}$ satisfying
$$ G_{\boldsymbol{\mu},\nu}(q) = \sum_{\boldsymbol{T} \in \mathrm{SSTab}(\boldsymbol{\mu},\nu)}
                                     q^{n \alpha(\boldsymbol{T}) + \mathrm{maj}(\boldsymbol{T}) + d_{\boldsymbol{\mu}}} , $$
  where $ d_{\boldsymbol{\mu}} = \mathrm{min} \{ \mathrm{Inv}(\boldsymbol{T})| \boldsymbol{T} \in \mathrm{SSTab}(\boldsymbol{\mu}) \} $
  is the minimum Inversion number on the set of semistandard tableaux of shape $\boldsymbol{\mu}$.

\vspace{1em}

\rm{Next}, we explain the representation theoretical meaning of Theorem.B.

Let $$ \boldsymbol{\mu} = (\mu^{(0)},\mu^{(1)},\cdots,\mu^{(n-1)}) $$ be an $n$-tuple of Young diagrams.
The LLT polynomial $G_{\boldsymbol{\mu}}(x;q)$ is symmetric in the variables $x$(\cite{LLT}).
Define a $q$-analog of Littlewood-Richardson coefficients by
$$ G_{\boldsymbol{\mu},\nu}(x;q) \coloneqq
  \sum_{\nu \vdash n} \widetilde{LR}_{\mu^{(0)},\cdots,\mu^{(n-1)}}^{\nu}(q) s_\nu(x) \,\,, $$
where $s_\nu(x)$ is the Schur function corresponding to the partition $\nu$.
Then
$\widetilde{LR}_{\mu^{(0)},\cdots,\mu^{(n-1)}}^{\nu}(q)$ has non-negative coefficients\cite{L}.

This positivity is deduced from the deep connection of LLT polynomials to the Fock space representation of
    the quantum affine algebra $U_q(\hat{sl}_n)$ and Kazhdan-Lusztig polynomials\cite{L}.

From this positivity we can consider Inversion number as a statistics on each irreducible component of
  the $GL_N$-module $V_{\mu^{(0)}} \otimes \cdots \otimes V_{\mu^{(n-1)}}$.

Again, we consider the case where $\mu^{(0)}=\mu^{(1)} = \cdots = \mu^{(n-1)} = \mu $, i.e.
$$ \boldsymbol{\mu} = \underbrace{(\mu,\cdots,\mu)}_{n} , $$
 and discuss a $q$-analog of the multiplicities in the submodule 
 $V_\mu^{\otimes n}[\zeta_n^i] \subset V_\mu^{\otimes n}$; 
 the $\zeta_n^i$-eigenspace of the action of a Coxeter element of the symmetric group $\mathfrak{S}_n$,   
 where $\zeta_n$ is a primitive root of unity.

\vspace{1em}

Let $G_{\boldsymbol{\mu}}(q;x)=\sum_{j \geq 0} a_j(x) q^j$ be the LLT polynomial.
For $0 \leq i \leq n-1$, set
$$ G_{\boldsymbol{\mu}}^{(i)}(x;q) \coloneqq \sum_{j \geq 0} a_{jn+i+d_{\boldsymbol{\mu}}}(x) q^{jn+i+d_{\boldsymbol{\mu}}} .$$
From Theorem.B, we can regard this polynomial $G_{\boldsymbol{\mu}}^{(i)}(x;q)$ as a graded character of $V_\mu^{\otimes n}[\zeta_n^i]$, 
i.e. the coefficient of $s_\nu(x)$ in $G_{\boldsymbol{\mu}}^{(i)}(x;q)$ gives 
a $q$-analog of a multiplicity in $V_\mu^{\otimes n}[\zeta_n^i]$.

\vspace{1em}

This paper is organized as follows.
In section 2, we prepare notations.
In section 3, we review the definition of LLT polynomials by \cite{HH1} 
   , positivity of $\widetilde{LR}_{\mu^{(0)},\cdots,\mu^{(n-1)}}^{\nu}(q)$
   , and the definition of a $q$-analog of Littlewood-Richardson coefficients by \cite{LLT}.
In section 4, we prove the $q$-multinomial expansion of LLT coefficients.

In section 5, we prove Theorem.B
In section 6, we discuss a $q$-analog of multiplicities in $V_\mu^{\otimes n}[\zeta_n^i]$.
In section 7, we discuss a $q$-analog of plethysm multiplicities.


%
%

\section{Preliminaries}

We start with introducing the notations.

\subsection{$q$-integers}

We define the $q$-integer $[k]_{q}$ by

\begin{equation*}
[k]_{q}= \frac{1-q^k}{1-q}\,\,\,,
\end{equation*}

and the $q$-binomial coefficients and the $q$-multinomial coefficients by

\begin{equation*}
\left[ \begin{array}{c}
n \\ k
\end{array}
\right]_{q} = \frac{[n]_{q}!}{[k]_{q}![n-k]_{q}!} \,\,\,, 
\end{equation*}

\begin{equation*}
\left[ \begin{array}{c}
n \\ k_1,k_2,\cdots,k_r
\end{array}
\right]_{q} = \frac{[n]_{q}!}{[k_1]_{q}![k_2]_{q}!\cdots[k_r]_{q}!} \,, 
\end{equation*}

where $ [k]_{q}!= [k]_{q}[k-1]_{q} \cdots [1]_{q} $ and $ k_1+k_2+ \cdots +k_r=n. $

\subsection{Young diagram}

A \textit{partition} of $n$ is a non-increasing sequences of non-negative integers summing to $n$.
We write $\lambda \vdash n$ or $| \lambda |=n$ if $\lambda$ is a partition of $n$.
And we use the same notation $\lambda$ to represent the Young diagram of size $n$ corresponding to $\lambda$.
The $\it{content}$ of a cell $u=(i,j)$ in a Young diagram $\lambda$ is the integer $c(u)=j-i$.

For a given Young diagram $\lambda$, a \textit{Young tableau} of shape $\lambda$ is
  a map from the set of cells (in the Young diagram $\lambda$) to the set of positive integers.
A \textit{semistandard} tableau is a Young tableau whose entries increase weakly
  along the rows and increase strictly down the columns.
We denote the set of semistandard tableaux of shape $\lambda$ by SSTab$(\lambda)$.
For a Young tableau $T$, the \it{weight} \rm{of} $T$ is the sequence
  $wt(T)=(\nu_1,\nu_2,\cdots)$, where $\nu_k$ is the number of entries equal to $k$.
We denote by SSTab$(\lambda,\nu)$ the set of semistandard tableaux of shape $\lambda$ with weight $\nu$.
For a tableau $T \in \mathrm{SSTab}(\lambda,\nu)$, we define its monomial $x^T=x^{wt(T)}=x_1^{\nu_1}x_2^{\nu_2}\cdots$.

Let  
$$ \boldsymbol{\mu} = (\mu^{(0)},\mu^{(1)},\cdots,\mu^{(n-1)}) $$
be an $n$-tuple of Young diagrams.
Let $\mathrm{SSTab}(\boldsymbol{\mu})$ be the set of tuples of semistandard tableaux
  $\boldsymbol{T}=(T_0,\cdots,T_{n-1})$, where $T_i$ has shape $\mu^{(i)}$.
Given $\boldsymbol{T}=(T_0,T_1,\cdots,T_{n-1}) \in \mathrm{SSTab}(\boldsymbol{\mu})$, 
we define its weight and monomial as 
$$ wt(\boldsymbol{T}) = wt(T_0) + \cdots +wt(T_{n-1})\,\,\,,\,\,\,
   x^{\boldsymbol{T}} = x^{T_0} \cdots x^{T_{n-1}}. $$ 

\vspace{1em}

\subsection{a total order on $SSTab(\mu)$}

Next we define a total order in $\mathrm{SSTab}(\mu)$.
For a given semistandard tableau $T$, by reading $T$ from left to right in consecutive rows,
  starting from top to bottom, we obtain the word $word(T)$.
We define a total order $>$ on the set of words (in which entry is a positive integer) as the lexicographic order.

\begin{definition}
We define a total order $>$ on $\mathrm{SSTab}(\mu)$ as follows ;
\par Let $T,U \in \mathrm{SSTab}(\mu)$.
\par We define $T>U$ if $word(T)>word(U)$.
\label{totalorder}
\end{definition}

For $T \in \mathrm{SSTab}(\boldsymbol{\mu})$,
we define maj$(T)$ with respect to this totally order on $\mathrm{SSTab}(\mu)$.

\begin{example}
\rm{Let} $T_1= 
\unitlength 0.1in
\begin{picture}(  6 ,  0 )(  2 , -4 )
%
\special{pn 8}%
\special{pa 200 600}%
\special{pa 400 600}%
\special{pa 400 200}%
\special{pa 200 200}%
\special{pa 200 600}%
\special{fp}%
%
\special{pn 8}%
\special{pa 200 400}%
\special{pa 800 400}%
\special{pa 800 200}%
\special{pa 200 200}%
\special{pa 200 400}%
\special{fp}%
%
\special{pn 8}%
\special{pa 600 400}%
\special{pa 600 200}%
\special{fp}%
\put(3.0000,-3.0000){\makebox(0,0){1}}%
\put(5.0000,-3.0000){\makebox(0,0){1}}%
\put(7.0000,-3.0000){\makebox(0,0){2}}%
\put(3.0000,-5.0000){\makebox(0,0){3}}%
\end{picture}%
 \,\,,\, T_2= 
\unitlength 0.1in
\begin{picture}(  6 ,  4 )(  2 , -4)
%
\special{pn 8}%
\special{pa 200 600}%
\special{pa 400 600}%
\special{pa 400 200}%
\special{pa 200 200}%
\special{pa 200 600}%
\special{fp}%
%
\special{pn 8}%
\special{pa 200 400}%
\special{pa 800 400}%
\special{pa 800 200}%
\special{pa 200 200}%
\special{pa 200 400}%
\special{fp}%
%
\special{pn 8}%
\special{pa 600 400}%
\special{pa 600 200}%
\special{fp}%
\put(3.0000,-3.0000){\makebox(0,0){1}}%
\put(5.0000,-3.0000){\makebox(0,0){1}}%
\put(7.0000,-3.0000){\makebox(0,0){2}}%
\put(3.0000,-5.0000){\makebox(0,0){4}}%
\end{picture}%
 \,\,,\, T_3=
\unitlength 0.1in
\begin{picture}(  7 ,  6 )(  2, -4 )
%
\special{pn 8}%
\special{pa 200 600}%
\special{pa 400 600}%
\special{pa 400 200}%
\special{pa 200 200}%
\special{pa 200 600}%
\special{fp}%
%
\special{pn 8}%
\special{pa 200 400}%
\special{pa 800 400}%
\special{pa 800 200}%
\special{pa 200 200}%
\special{pa 200 400}%
\special{fp}%
%
\special{pn 8}%
\special{pa 600 400}%
\special{pa 600 200}%
\special{fp}%
\put(3.0000,-3.0000){\makebox(0,0){1}}%
\put(7.0000,-3.0000){\makebox(0,0){2}}%
\put(5.0000,-3.0000){\makebox(0,0){2}}%
\put(3.0000,-5.0000){\makebox(0,0){2}}%
\end{picture}%
$,
   and $\boldsymbol{T}=(T_3,T_1,T_2,T_1)$.

\vspace{1em}

Then $word(T_1)=1123,\, word(T_2)=1124,\,word(T_3)=1222$.
Thus $T_1<T_2<T_3$.
\end{example}

\subsection{Major index and Foata map.}

\begin{definition}
Let $\Gamma$ be a totally ordered set.
Let $w=(w_1, \cdots ,w_n)$ be a n-tuple of elements in $\Gamma$.
A decent of $w$ is an integer $i$, $1 \leq i \leq n-1$, for which $w_i>w_{i+1}$.
A major index of $w$ is defined to be the sum of its descents, i.e.

\begin{equation*}
\mathrm{maj}(w)=\sum_{i, w_i>w_{i+1}} i .
\end{equation*}

\label{major}
\end{definition}

\begin{example}
In the above example, 
$\mathrm{maj}(\boldsymbol{T})=1+3=4$.
\end{example}

A \textit{word} $w=w_1 w_2 \cdots$ is a sequence of positive integers,
  and its \textit{weight} is the sequence $\mathrm{wt}(w)=(\nu_1,\nu_2,\ldots)$, where $\nu_k$ is the number of $w_i$ equal to $k$.
Let $\nu$ be a composition of $n$; namely $\nu=(\nu_1,\nu_2,\ldots,)$ is a sequence of non-negative integers summing to $n$.
Let $\mathrm{Word}(\nu)$ be the set of words with weight $\nu$.
There are two statistics on the set $\mathrm{Word}(\nu)$; inversion number inv and major index maj.
It is known that these two statistics have the same distribution on $\mathrm{Word}(\nu)$.

\begin{theorem}[cf.\cite{F},\cite{Lo}]
\par Let $\nu$ be a composition of $n$.
Then there exists a bijection $\Phi \colon \mathrm{Word}(\nu) \rightarrow \mathrm{Word}(\nu)$ satisfying

\begin{equation*}
\mathrm{inv}(\Phi (w)) = \mathrm{maj}(w)  ,\,\,\,  (w \in \mathrm{Word}(\nu)). 
\end{equation*}

\par In particular,
\begin{align*}
\sum_{w \in \mathrm{Word}(\nu)} q^{\mathrm{maj}(w)} &= \sum_{w \in \mathrm{Word}(\nu)} q^{\mathrm{inv}(w)} \\
  &= \biggl[ \begin{array}{c}
n \\ \nu
\end{array}
\biggr]_q.
\end{align*}
\label{Foata}
\end{theorem}

We can construct the bijection $\Phi$ explicitly, and call it Foata map(see \cite{Lo}).

\subsection{Kostka coefficients and inverse Kostka coefficients}

Finally, we recall the Kostka coefficients and the inverse Kostka coefficients.
Let $s_\lambda$, and $m_\lambda,$ denote the Schur function, and the monomial symmetric function
corresponding to $\lambda$, respectively.

\begin{definition}
For $\lambda,\mu \vdash n$, the Kostka coefficient $K_{\lambda,\mu}$ is defined by
$$ s_\lambda = \sum_{\mu \vdash n}K_{\lambda,\mu}m_\mu.$$
\par Similarly the inverse Kostka coefficient $K_{\lambda,\mu}^{-1}$ is given by
$$ m_\lambda = \sum_{\mu \vdash n}K_{\lambda,\mu}^{-1}s_\mu.$$
\label{inverseKostka}
\end{definition}

%
%

\section{LLT polynomials and a q-analog of Littlewood-Richardson coefficients}

\subsection{LLT polynomials} 

Let $n$ be a positive integer and $N$ be a sufficiently large integer.
Let $$ \boldsymbol{\mu} = (\mu^{(0)},\mu^{(1)},\cdots,\mu^{(n-1)}) $$
be an $n$-tuple of Young diagrams.
For a $\boldsymbol{T}=(T_0 , T_1 ,\cdots,T_{n-1}) \in \mathrm{SSTab}(\boldsymbol{\mu})$, 
two entries $T^{(i)}(u) > T^{(j)}(v)$ form an \textit{Inversion} if either
\begin{align*}
(\mathrm{i} ). \,\, i<j  \text{ and } c(u)=c(v), \text{ or} \\
(\mathrm{ii}). \,\, i>j  \text{ and } c(u)=c(v)-1 .
\end{align*}
Denote by Inv$(\boldsymbol{T})$ the number of inversions in $\boldsymbol{T}$.

\begin{example}
\rm{Let} 
$\boldsymbol{T} = \Bigl( \,\, 
\unitlength 0.1in
\begin{picture}(  4,  0)(  2, -4.5 )
%
\special{pn 8}%
\special{pa 200 600}%
\special{pa 600 600}%
\special{pa 600 200}%
\special{pa 200 200}%
\special{pa 200 600}%
\special{fp}%
%
\special{pn 8}%
\special{pa 200 400}%
\special{pa 600 400}%
\special{fp}%
%
\special{pn 8}%
\special{pa 400 600}%
\special{pa 400 200}%
\special{fp}%
\put(3.0000,-3.1000){\makebox(0,0){2}}%
\put(3.0000,-5.1000){\makebox(0,0){4}}%
\put(4.9000,-5.1000){\makebox(0,0){4}}%
\put(4.8000,-3.1000){\makebox(0,0){2}}%
\end{picture}%
 \,\,,\,\, 
\unitlength 0.1in
\begin{picture}(  4,  0 )(  2 , -4.5)
%
\special{pn 8}%
\special{pa 200 600}%
\special{pa 600 600}%
\special{pa 600 200}%
\special{pa 200 200}%
\special{pa 200 600}%
\special{fp}%
%
\special{pn 8}%
\special{pa 400 600}%
\special{pa 400 200}%
\special{fp}%
%
\special{pn 8}%
\special{pa 200 400}%
\special{pa 600 400}%
\special{fp}%
\put(3.0000,-3.0000){\makebox(0,0){1}}%
\put(5.0000,-3.0000){\makebox(0,0){1}}%
\put(3.0000,-5.0000){\makebox(0,0){2}}%
\put(5.0000,-5.0000){\makebox(0,0){3}}%
\end{picture}%
 \,\, \Bigr) \in \mathrm{SSTab}((2\,2),(2\,2)).$
\vspace{1em}

Then $\boldsymbol{T}$ has $5$ inversions of type $($i$)$ and $1$ inversion of type $($ii$)$.
Therefore Inv$(\boldsymbol{T})=6$.
\end{example}

\begin{definition}$($\cite{HH1}$)$
\par
The \textit{LLT polynomial} indexed by $\boldsymbol{\mu}$ is
$$ G_{\boldsymbol{\mu}}(x;q) = 
      \sum_{\boldsymbol{T} \in \mathrm{SSTab}(\boldsymbol{\mu})} q^{\mathrm{Inv}(\boldsymbol{T})} x^{\boldsymbol{T}} .$$
Also, we denote $G_{\boldsymbol{\mu},\nu}(q)$ by the coefficient of 
$x^{\nu}=x_1^{\nu_1}x_2^{\nu_2}\cdots$ in the LLT polynomial $G_{\boldsymbol{\mu}}(x;q)$, and call it LLT coefficient.
\end{definition}

\begin{example}
\rm{Let} $\boldsymbol{\mu}=((2),(2))$ and $\nu=(2\,2)$.
$\mathrm{SSTab}(\boldsymbol{\mu},\nu)$ consists of $3$ tableaux; 
$$ \boldsymbol{T}_1= 
\unitlength 0.1in
\begin{picture}( 12 , -1 )(  2.55 , -3.5 )
%
\special{pn 8}%
\special{pa 400 400}%
\special{pa 800 400}%
\special{pa 800 200}%
\special{pa 400 200}%
\special{pa 400 400}%
\special{fp}%
%
\special{pn 8}%
\special{pa 600 400}%
\special{pa 600 200}%
\special{fp}%
%
\special{pn 8}%
\special{pa 1000 400}%
\special{pa 1400 400}%
\special{pa 1400 200}%
\special{pa 1000 200}%
\special{pa 1000 400}%
\special{fp}%
%
\special{pn 8}%
\special{pa 1200 400}%
\special{pa 1200 200}%
\special{fp}%
\put(3.0000,-3.0000){\makebox(0,0){(}}%
\put(15.0000,-3.0000){\makebox(0,0){)}}%
\put(9.0000,-3.3000){\makebox(0,0){,}}%
\put(5.0000,-3.0000){\makebox(0,0){1}}%
\put(7.0000,-3.0000){\makebox(0,0){1}}%
\put(11.0000,-3.0000){\makebox(0,0){2}}%
\put(13.0000,-3.0000){\makebox(0,0){2}}%
\end{picture}%
 \,\,\,,\, \boldsymbol{T}_2= 
\unitlength 0.1in
\begin{picture}( 12 , -1 )(  2.55 , -3.5 )
%
\special{pn 8}%
\special{pa 400 400}%
\special{pa 800 400}%
\special{pa 800 200}%
\special{pa 400 200}%
\special{pa 400 400}%
\special{fp}%
%
\special{pn 8}%
\special{pa 600 400}%
\special{pa 600 200}%
\special{fp}%
%
\special{pn 8}%
\special{pa 1000 400}%
\special{pa 1400 400}%
\special{pa 1400 200}%
\special{pa 1000 200}%
\special{pa 1000 400}%
\special{fp}%
%
\special{pn 8}%
\special{pa 1200 400}%
\special{pa 1200 200}%
\special{fp}%
\put(3.0000,-3.0000){\makebox(0,0){(}}%
\put(15.0000,-3.0000){\makebox(0,0){)}}%
\put(9.0000,-3.3000){\makebox(0,0){,}}%
\put(5.0000,-3.0000){\makebox(0,0){1}}%
\put(13.0000,-3.0000){\makebox(0,0){2}}%
\put(7.0000,-3.0000){\makebox(0,0){2}}%
\put(11.0000,-3.0000){\makebox(0,0){1}}%
\end{picture}%
 \,\,\,,\, 
   \boldsymbol{T}_3= 
\unitlength 0.1in
\begin{picture}( 12 , -1 )(  2.55 , -3.5 )
%
\special{pn 8}%
\special{pa 400 400}%
\special{pa 800 400}%
\special{pa 800 200}%
\special{pa 400 200}%
\special{pa 400 400}%
\special{fp}%
%
\special{pn 8}%
\special{pa 600 400}%
\special{pa 600 200}%
\special{fp}%
%
\special{pn 8}%
\special{pa 1000 400}%
\special{pa 1400 400}%
\special{pa 1400 200}%
\special{pa 1000 200}%
\special{pa 1000 400}%
\special{fp}%
%
\special{pn 8}%
\special{pa 1200 400}%
\special{pa 1200 200}%
\special{fp}%
\put(3.0000,-3.0000){\makebox(0,0){(}}%
\put(15.0000,-3.0000){\makebox(0,0){)}}%
\put(9.0000,-3.3000){\makebox(0,0){,}}%
\put(7.0000,-3.0000){\makebox(0,0){2}}%
\put(11.0000,-3.0000){\makebox(0,0){1}}%
\put(5.0000,-3.0000){\makebox(0,0){2}}%
\put(13.0000,-3.0000){\makebox(0,0){1}}%
\end{picture}%
 \,\,.$$
The Inversion numbers are Inv$(\boldsymbol{T}_1)=1$, Inv$(\boldsymbol{T}_2)=0$, and Inv$(\boldsymbol{T}_3)=2$, respectively.
So, 
$$ G_{\boldsymbol{\mu},\nu}(q) = 1+q+q^2 .$$
\end{example}

\begin{theorem}$($\cite[Theorem.6.1.]{LLT}, \cite[Appendix]{HH2}$)$
\par
The LLT polynomial $G_{\boldsymbol{\mu}}(x;q)$ is symmetric in the variables $x$.
\label{symmetry}
\end{theorem}

$\bf{Remark}$.
Lascoux, Leclerc, and Thibon originally defined LLT polynomials by using combinatorial objects, 
  \textit{ribbon tableaux} and their \textit{spins}\cite{LLT}.
The relationship between $G_{\boldsymbol{\mu}}(x;q)$ and the original polynomial $L_{\widetilde{\mu},\nu}(q)$ is as follows.
Let $\widetilde{\mu}$ be the Young diagram whose $n$-quotient is $\boldsymbol{\mu}$, 
  $Tab_n(\widetilde{\mu})$ be the set of $n$-ribbon tableaux of shape $\widetilde{\mu}$, 
  and $L_{\widetilde{\mu},\nu}(q) \coloneqq \sum_{T \in Tab_n(\widetilde{\mu})} q^{\mathrm{spin}(T)} , \text{ \cite[eq(39)]{L}}.$

Then 
$$G_{\boldsymbol{\mu},\nu}(q^2)=q^{b_{\widetilde{\mu}}}L_{\widetilde{\mu},\nu}(q^{-1}),$$
where $b_{\widetilde{\mu}}=\mathrm{max} \{ \mathrm{spin}(T)|T \in Tab_n(\widetilde{\mu}) \}$ is the maximal spin of
  the set of $n$-ribbon tableaux of shape $\widetilde{\mu}$.
See \cite{L} and \cite[$\S 5$]{HH1} for more details.

%
%

\subsection{a $q$-analog of Littlewood-Richardson coefficients}

In this section we review a $q$-analog of Littlewood-Richardson coefficients\cite{LLT}.
Again, let $$ \boldsymbol{\mu} = (\mu^{(0)},\mu^{(1)},\cdots,\mu^{(n-1)}) $$
be an $n$-tuple of Young diagrams.

\begin{definition}[a $q$-analog of Littlewood-Richardson coefficients]
\par
Define a $q$-analog of Littlewood-Richardson coefficients by
$$ G_{\boldsymbol{\mu},\nu}(x;q) \coloneqq
  \sum_{\nu \vdash n} \widetilde{LR}_{\mu^{(0)},\cdots,\mu^{(n-1)}}^{\nu}(q) s_\nu(x) \,\,, $$
where $s_\nu(x)$ is the Schur function corresponding to the partition $\nu$.
\label{q-LR}
\end{definition}

Note that
$$ \widetilde{LR}_{\mu^{(0)},\cdots,\mu^{(n-1)}}^{\nu}(q) =
  \sum_{\rho \vdash n}K_{\rho,\nu}^{-1}G_{\boldsymbol{\mu},\rho}(q) \,\,, $$
where $K_{\rho,\nu}^{-1}$ is the inverse Kostka coefficient defined in Definition.\ref{inverseKostka}.

From the definition,
$G_{\boldsymbol{\mu}}(x;1)$ is the character of the $GL_N$-module 
$V_{\mu^{(0)}} \otimes \cdots \otimes V_{\mu^{(n-1)}}$,
where $V_{\lambda}$ denotes the irreducible representation of the general linear group $GL_N$
   corresponding to $\lambda$.
So, we can consider $ \widetilde{LR}_{\mu^{(0)},\cdots,\mu^{(n-1)}}^{\nu}(q)$
   as a $q$-analog of the Littlewood-Richardson coefficients (they are the tensor product multiplicities).

\begin{theorem}$($\cite{L}$)$
\par
$$\widetilde{LR}_{\mu^{(0)},\cdots,\mu^{(n-1)}}^{\nu}(q) \in \mathbb{Z}_{\geq 0}[q]. $$
\label{positivity}
\end{theorem}

$\bf{Remark}$.
Theorem.\ref{positivity} is deduced from the deep connection of LLT polynomials to the Fock space representation of
    the quantum affine algebra $U_q(\hat{sl}_n)$ and Kazhdan-Lusztig polynomials\cite{L}.
These $q$-Littlewood-Richardson coefficients $\widetilde{LR}_{\mu^{(0)},\cdots,\mu^{(n-1)}}^{\nu}(q)$ appear as
 coefficients of grobal basis of the Fock space representation of $U_q(\hat{sl}_n)$.

In the notation of \cite{L} and the remark after Theorem.\ref{symmetry}, we obtain the followings
$($See \cite[eq(41)]{L}$)$.
$$ e_{n \nu,\widetilde{\mu}}(q)=q^{b_{\widetilde{\mu}}}
    \widetilde{LR}_{\mu^{(0)},\cdots,\mu^{(n-1)}}^{\nu}(q^{-2}) .$$

From Theorem.\ref{positivity} we can consider Inv as a statistics on each irreducible component of
  the $GL_N$-module $V_{\mu^{(0)}} \otimes \cdots \otimes V_{\mu^{(n-1)}}$.

%
%

\section{$q$-multinomial expansion of LLT coefficients}

We prove a $q$-multinomial expansion of LLT polynomials in this section.

Throughout this section, we consider the case where $\mu^{(0)}=\mu^{(1)} = \cdots = \mu^{(n-1)} = \mu $, i.e.
\begin{equation*}
\boldsymbol{\mu} = \underbrace{(\mu,\cdots,\mu)}_{n} .
\end{equation*}

\begin{proposition}
Let $\boldsymbol{\mu}=(\underbrace{\mu,\cdots,\mu})$ and $\nu$ be a partition of $n |\mu|$.
Then there exists a bijection
  $\mathcal{F} \colon \mathrm{SSTab}(\boldsymbol{\mu},\nu) \rightarrow \mathrm{SSTab}(\boldsymbol{\mu},\nu)$ satisfying
 \begin{equation}
 \mathrm{Inv}(\boldsymbol{T}) = \mathrm{Inv}(T_0',\ldots,T_{n-1}') + \mathrm{inv}(\mathcal{F}(\boldsymbol{T})) ,
 \end{equation}
where $T_0' \leq \ldots \leq T_{n-1}'$ is the rearrangement of the parts of $\mathcal{F}(\boldsymbol{T})$, and
   $\mathrm{inv}(\mathcal{F}(\boldsymbol{T}))$ is the inversion number with respect to
   the total order on $\mathrm{SSTab}(\mu)($see Definition.\ref{totalorder}.$)$
\label{keylem}
\end{proposition}

\begin{example}
\rm{Let} $\mu=(3)$, $n=5$, $\boldsymbol{T}=(233,222,112,134,133)$,
  then $\mathcal{F}\boldsymbol{T}=(234,233,112,133,122)$ (see Example.\ref{exampleF} below).
Hence $\boldsymbol{T}'=(112,122,133,233,234)$.
Since Inv$(\boldsymbol{T})=16$, Inv$(\boldsymbol{T}')=8$, and inv$(\mathcal{F}\boldsymbol{T})=8$,
  we have $\mathrm{Inv}(\boldsymbol{T}) = \mathrm{Inv}(\boldsymbol{T}') + \mathrm{inv}(\mathcal{F}\boldsymbol{T}).$
\end{example}

\vspace{1em}

In particular, from Theorem.\ref{Foata} we obtain a $q$-multinomial expansion of LLT coefficients.

\begin{theorem}[$q$-multinomial expansion of LLT coefficients]
Let $\boldsymbol{\mu}=(\underbrace{\mu,\ldots,\mu}_n)$.
Then
$$ G_{\boldsymbol{\mu},\nu}(q)=\sum_{(T_0 \leq \cdots \leq T_{n-1}) \in \mathrm{SSTab}(\boldsymbol{\mu},\nu)}
    q^{\mathrm{Inv}(T_0 \leq \cdots \leq T_{n-1})} \biggl[ \begin{array}{c} n \\ \rho \end{array} \biggr]_q,  $$
where $\rho$ is the composition which determined from $T_0 \leq \cdots \leq T_{n-1}$;
  i.e. $T_0=\cdots=T_{\rho_1-1} < T_{\rho_1}=\cdots=T_{\rho_1+\rho_2-1} < T_{\rho_1+\rho_2}= \cdots$.
\label{fermionicformula}
\end{theorem}

\begin{example}
\rm{Let} $n=3$, $\mu=(2)$ and $\nu=(3 \, 2 \, 1)$.
Then
$$\{ \boldsymbol{T} \in \mathrm{SSTab}(\boldsymbol{\mu},\nu) \,|\, T_0 \leq T_1 \leq T_2 \} = \{ (11,12,23),(11,13,22),(12,12,13) \}.$$
Thus,
$$ G_{\boldsymbol{\mu},\nu}(q) =
   q   \biggl[ \begin{array}{c} 3 \\ 1,1,1 \end{array} \biggr]_q
 + q^2 \biggl[ \begin{array}{c} 3 \\ 1,1,1 \end{array} \biggr]_q
 +     \biggl[ \begin{array}{c} 3 \\ 2,1 \end{array} \biggr]_q  . $$
\end{example}

\vspace{1em}

We use the following proposition which gives an answer in the case of $n=2$.

\begin{proposition}
\cite[Lemma.8.5. and its Proof]{CL}
There is an involution $\phi \colon \mathrm{SSTab}(\mu)^2 \rightarrow \mathrm{SSTab}(\mu)^2$ such that
$$ \mathrm{Inv}(\boldsymbol{T})=
 \begin{cases}
   \mathrm{Inv}(\boldsymbol{T})+1 & \text{if $T_0 < T_1$}, \\
   \mathrm{Inv}(\boldsymbol{T}) & \text{if $T_0 = T_1$}, \\
   \mathrm{Inv}(\boldsymbol{T})-1 & \text{if $T_0 > T_1$},
 \end{cases} $$
where $\boldsymbol{T}=(T_0,T_1) \in \mathrm{SSTab}(\mu)^2$.
\label{labyrinth}
\end{proposition}

$\bf{Remarks}$.
$\bf{1}$. $\phi$ modifies some part in $T_1$ and $T_2$ including the first cell $u \in \mu$ such that $T_0(u) \not= T_1(u)$(\cite{CL}).
In particular, if the length of $\nu$ is equal to $2$ then $\phi(T_0,T_1)=(T_1,T_0)$.

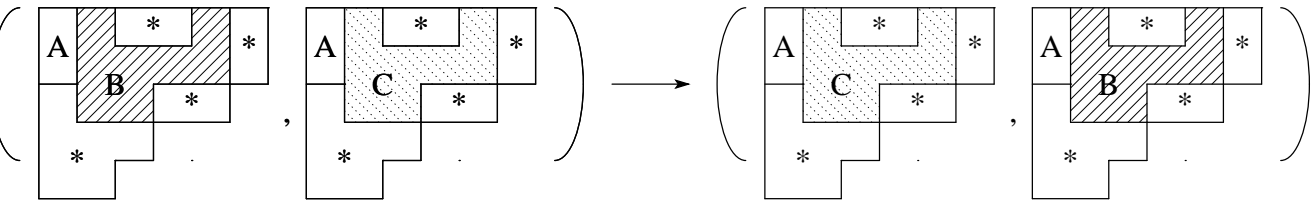
\begin{center}
\hspace{3em}
\input{2008fig13.tex}
\end{center}

\vspace{1em}

$\bf{2}$. Let $\phi(T_0,T_1)=(\widetilde{T_0},\widetilde{T_1})$.
Then $$
T_0 < T_1 \, \Longrightarrow \widetilde{T_0} > \widetilde{T_1} \,\,, \text{ and } \,\,
T_0 > T_1 \, \Longrightarrow \widetilde{T_0} < \widetilde{T_1} .$$

$\bf{3}$. $\phi$ preserves the weight of $\boldsymbol{T}$.

\vspace{1em}

Now we define an operation $\circledast$ to construct $\mathcal{F}$.
This operation just specifies an order of adjacent transpositions to get a new word $\widetilde{w}$.
(See also Example.\ref{exam-operation*} below.)

$\bf{Operation}$ $\boldsymbol{\circledast}$

Let $w=w_0 \cdots w_{n-1}$ be a word 
 and $v_1 < v_2 < \cdots < v_r$ be the letters which is appears $w$.

(i). Set $j_1=\mathrm{min} \{j \,|\, w_j=v_1 \,\}$.
We move $w_{j_1}$ to the first position by adjacent transpositions $s_i=(i-1 ,\, i)$.

(ii). Set $j_2=\mathrm{min} \{j>j_1 \,|\, w_j=v_1 \,\}$,
        and we move $w_{j_2}$ to the second position by adjacent transpositions and so on.
Then $w$ is transformed into $w(v_1) = v_1 \cdots v_1 w'$.

(iii). Apply (i) and (ii) in case $v_2$ and $w'$.
Then $w$ is transformed into $w(v_2) = v_1 \cdots v_1 v_2 \cdots v_2 w''$.

(iv). Repeat (iii) to get a new word $\widetilde{w}=\widetilde{w_0}\cdots\widetilde{w_{n-1}}$.
     (Thus $\widetilde{w_0} \leq \cdots \leq \widetilde{w_{n-1}}$.)

\begin{example}
\rm{Let} $w=213312$.($v_1=1,v_2=2$ and $v_3=3$)
Then 
\begin{align*}
w=213312 & \overset{s_1}{\rightarrow} 123312 \\
         & \overset{s_4}{\rightarrow}\overset{s_3}{\rightarrow}\overset{s_2}{\rightarrow} 112332 \\
         & \overset{s_5}{\rightarrow}\overset{s_4}{\rightarrow} 112233 = \widetilde{w}.
\end{align*}
\label{exam-operation*}
\end{example}

\vspace{1em}

Now we construct the bijection $\mathcal{F}$ to prove Proposition\ref{keylem}.

Let $\boldsymbol{T}=(T_0,\ldots,T_{n-1}) \in \mathrm{SSTab}(\mu)^n$.
We define $\phi_i(\boldsymbol{T})$ as follows;
$$ \phi_i(T) \coloneqq (T_0,\cdots,T_{i-2},\phi (T_{i-1},T_i),T_{i+1},\cdots,T_{n-1}). $$

$\bf{Definition}$ $\bf{of}$ $\boldsymbol{\mathcal{F}}$.
(See also Example.\ref{exampleF} below.)

Let $\boldsymbol{T}=(T_0,\ldots,T_{n-1}) \in \mathrm{SSTab}(\mu)^n$.

(i). Set $w(1,1)=T_1(1,1) \cdots T_n(1,1)$.
If $\widetilde{w(1,1)}=s_{i_k} \cdots s_{i_1} w(1,1)$ in the operation $\circledast$,
  then we denote $\boldsymbol{T}^{(1,1)}=\boldsymbol{\phi_{i_k} \cdots \phi_{i_1}} \boldsymbol{T}$.
(Note that $T^{(1,1)}_1(1,1) \leq \cdots \leq T^{(1,1)}_n(1,1).$)

(ii). Let $X_e= \{ T^{(1,1)}_j \,|\,T^{(1,1)}_j(1,1)=e \} $.
We repeat (i) on each $X_e$ with respect to the next cell ($(1,2)$ or $(2,1)$), 
 and we get $\boldsymbol{T}^{(1,2)}$(or $\boldsymbol{T}^{(2,1)}$).
(Note that we use the reading described before Definition.\ref{totalorder}.)

(iii). Repeat (ii) to the next cell.
i.e. Suppose we have $\boldsymbol{T}^{(v)}$ for a cell $v$.
Let $v'$ be the next cell and $\lambda(v)$ denotes the Young diagram which is consisted of the cells lying before $v'$. 
For a semistandard tableaux $e \in \mathrm{SSTab}(\lambda(v))$, 
 let $X_e= \{ T^{(v)}_j \,|\,T^{(v)}_j(a,b)=e(a,b) \text{ , } (a,b) \in \lambda(v) \}$.
We repeat (i) on each $X_e$ with respect to the cell $v'$, and we get $\boldsymbol{T}^{(v')}$.
(Note that we don't change the numbers which belong to $\lambda(v)$.)

(iv). Repeat (iii) until the last cell $u$, say $\boldsymbol{T}'=(T_0',\ldots,T_{n-1}')=\boldsymbol{T}^{(u)}$.
(Note that $T_0' \leq \cdots \leq T_{n-1}'.$)

(v). If $\boldsymbol{T}'=\boldsymbol{\phi_{i_r} \cdots \phi_{i_1}}\boldsymbol{T}$,
  then we set $\mathcal{F}\boldsymbol{T}=\boldsymbol{s_{i_1} \cdots s_{i_r}}\boldsymbol{T}'$.

\vspace{1em}

In the above situation we write
\begin{equation}
\boldsymbol{T}'= \phi_{i_r} \cdots \phi_{i_1} \boldsymbol{T} \hspace{1em} \text{in Def.$\mathcal{F}$}.
\end{equation}
For a cell $v \in \mu$, if $\boldsymbol{T}^{(v)}=\phi_{i_k} \cdots \phi_{i_1} \boldsymbol{T}$ in the definition of $\mathcal{F}$, 
we write
\begin{equation}
\boldsymbol{T}^{(v)}= \phi_{i_k} \cdots \phi_{i_1} \boldsymbol{T} \hspace{1em} \text{in Def.$\mathcal{F}$}.
\end{equation}

\begin{example}
\rm{Let} $\mu=(3)$, $n=5$, and $\boldsymbol{T}=(233,222,112,134,133)$.
Then
\begin{align*}
 \boldsymbol{T}=(233,222,112,134,133)
  & \overset{\phi_2}{\rightarrow} (233,112,222,134,133) \\
  & \overset{\phi_1}{\rightarrow} (112,233,222,134,133) \\
  & \overset{\phi_3}{\rightarrow} (112,233,122,234,133) \\
  & \overset{\phi_2}{\rightarrow} (112,133,222,234,133) \\
  & \overset{\phi_4}{\rightarrow} (112,133,222,134,233) \\
  & \overset{\phi_3}{\rightarrow}
      (\underbrace{112,133,122}_{X_1},\underbrace{234,233}_{X_2}) =\boldsymbol{T}^{(1,1)} \\
  & \overset{\phi_2}{\rightarrow}(112,122,133,\underbrace{234,233}_{X_{23}}) =\boldsymbol{T}^{(1,2)} \\
  & \overset{\phi_4}{\rightarrow}(112,122,133,233,234) =\boldsymbol{T}^{(1,3)}=\boldsymbol{T}'.
\end{align*}
Thus 
$\boldsymbol{T}^{(1,1)} = \phi_3 \phi_4 \phi_2 \phi_3 \phi_1 \phi_2 \boldsymbol{T}$ in Def.$\mathcal{F}$,
$\boldsymbol{T}^{(1,2)} = \phi_2 \phi_3 \phi_4 \phi_2 \phi_3 \phi_1 \phi_2 \boldsymbol{T}$ in Def.$\mathcal{F}$,
$\boldsymbol{T}' = \phi_4 \phi_2 \phi_3 \phi_4 \phi_2 \phi_3 \phi_1 \phi_2 \boldsymbol{T}$ in Def.$\mathcal{F}$, and
\begin{equation*}
\mathcal{F}\boldsymbol{T}=s_2s_1s_3s_2s_4s_3s_2s_4\boldsymbol{T}'=(234,233,112,133,122).
\end{equation*}
\label{exampleF}
\end{example}

Next we define the inverse of $\mathcal{F}$.

$\bf{Definition}$ $\bf{of}$ $\boldsymbol{\mathcal{G}}$.
(See also Example.\ref{exampleG} below.)

In the definition of $\mathcal{F}$, we exchange the role of $\boldsymbol{s_i}$'s(resp. $\boldsymbol{\phi_i}$'s) and
   $\boldsymbol{\phi_i}$'s(resp. $\boldsymbol{s_i}$'s).
More precisely, let $\boldsymbol{U} \in \mathrm{SSTab}(\boldsymbol{\mu})$.

(i). Set $w(1,1)=U_0(1,1) \cdots U_{n-1}(1,1)$.
If $\widetilde{w(1,1)}=s_{i_k} \cdots s_{i_1} w(1,1)$ in the operation $\circledast$,
  then we denote $\dot{\boldsymbol{U}}^{(1,1)}=\boldsymbol{s_{i_k} \cdots s_{i_1}} \boldsymbol{U}$.
(Note that $\dot{U}^{(1,1)}_0(1,1) \leq \cdots \leq \dot{U}^{(1,1)}_{n-1}(1,1).$)

(ii). Let $Y_e= \{ \dot{U}^{(1,1)}_j \,|\,\dot{U}^{(1,1)}_j(1,1)=e \} $.
We repeat (i) on each $Y_e$ with respect to the next cell ($(1,2)$ or $(2,1)$),
  and we get $\dot{\boldsymbol{U}}^{(1,2)}$(or $\dot{\boldsymbol{U}}^{(2,1)}$).

(iii). Repeat (ii) to the next cell.
i.e. Suppose we have $\dot{\boldsymbol{U}}^{(v)}$ for a cell $v$.
Let $v'$ be the next cell and $\lambda(v)$ denotes the Young diagram which is consisted of the cells lying before $v'$.
For a semistandard tableaux $e \in \mathrm{SSTab}(\lambda(v))$,
 let $Y_e= \{ U^{(v)}_j \,|\,U^{(v)}_j(a,b)=e(a,b) \text{ , } (a,b) \in \lambda(v) \}$.
We repeat (i) on each $Y_e$ with respect to the cell $v'$, and we get $\dot{\boldsymbol{U}}^{(v')}$.
(Note that we don't change the numbers which belong to $\lambda(v)$.)

(iv). Repeat (iii) until the last cell $u$, say $\dot{\boldsymbol{U}'}=(\dot{U_0'},\ldots,\dot{U_{n-1}'})=\dot{\boldsymbol{U}}^{(u)}$.
(Note that $\dot{\boldsymbol{U}'}$ is only the rearrangement of $\boldsymbol{U}$ in order.)

(v). If $\dot{\boldsymbol{U}'}=\boldsymbol{s_{i_r} \cdots s_{i_1}}\boldsymbol{U}$,
  then we set $\mathcal{G}\boldsymbol{U}=\boldsymbol{\phi_{i_1} \cdots \phi_{i_r}}\dot{\boldsymbol{U}'}$.

\vspace{1em}

In the above situation we write
\begin{equation}
\dot{\boldsymbol{U}}'= s_{i_r} \cdots s_{i_1} \boldsymbol{U} \hspace{1em} \text{in Def.$\mathcal{G}$}.
\end{equation}
For a cell $v \in \mu$, if $\dot{\boldsymbol{U}}^{(v)}=s_{i_k} \cdots s_{i_1} \boldsymbol{U}$ in the definition of $\mathcal{G}$,
we write
\begin{equation}
\dot{\boldsymbol{U}}^{(v)}= s_{i_k} \cdots s_{i_1} \boldsymbol{U} \hspace{1em} \text{in Def.$\mathcal{G}$}.
\end{equation}

\begin{example}
\rm{Let} $\mu=(3)$, $n=5$, and $U=(234,233,112,133,122)$.
Then
\begin{align*}
 \boldsymbol{U}=(234,233,112,133,122)
  & \overset{s_2}{\rightarrow}\overset{s_1}{\rightarrow}\overset{s_3}{\rightarrow}\overset{s_2}{\rightarrow}
    \overset{s_4}{\rightarrow}\overset{s_3}{\rightarrow}
      (\underbrace{112,133,122}_{Y_1},\underbrace{234,233}_{Y_2}) = \dot{\boldsymbol{U}}^{(1,1)}  \\
  & \overset{s_2}{\rightarrow}(112,122,133,\underbrace{234,233}_{Y_{23}}) =\dot{\boldsymbol{U}}^{(1,2)} \\
  & \overset{s_4}{\rightarrow}(112,122,133,233,234) = \dot{\boldsymbol{U}}^{(1,3)} = \dot{\boldsymbol{U}'}.
\end{align*}

Thus 
$\dot{\boldsymbol{U}}^{(1,1)} = s_3 s_4 s_2 s_3 s_1 s_2 \boldsymbol{U}$ in Def.$\mathcal{G}$,
$\dot{\boldsymbol{U}}^{(1,2)} = s_2 s_3 s_4 s_2 s_3 s_1 s_2 \boldsymbol{U}'$ in Def.$\mathcal{G}$,
$\dot{\boldsymbol{U}}' = s_4 s_2 s_3 s_4 s_2 s_3 s_1 s_2 \boldsymbol{U}'$ in Def.$\mathcal{G}$, and
\begin{equation*}
\mathcal{G}\boldsymbol{T} = \phi_2 \phi_1 \phi_3 \phi_2 \phi_4 \phi_3 \phi_2 \phi_4 \dot{\boldsymbol{U}}'=(233,222,112,134,133).
\end{equation*}
\label{exampleG}
\end{example}

\begin{lemma}
Let $\boldsymbol{T} \in \mathrm{SSTab}(\boldsymbol{\mu})$ and 
$\boldsymbol{T}'=(T_0'\leq \cdots \leq T_{n-1}')$ be the rearrangement of the parts of $\mathcal{F}\boldsymbol{T}$.
If 
\begin{equation}
\boldsymbol{T}'= \phi_{i_r} \cdots \phi_{i_1} \boldsymbol{T} \hspace{1em} \text{in Def.$\mathcal{F}$},
\label{eq.1}
\end{equation}
then
\begin{equation}
\boldsymbol{T}'= s_{i_r} \cdots s_{i_1} (\mathcal{F}\boldsymbol{T}) \hspace{1em} \text{in Def.$\mathcal{G}$}.
\label{eq.2}
\end{equation}
\label{lem-bijection}
\end{lemma}

\begin{proof}

Let $\hat{\mu}$ be the Young diagram obtained by removing the last cell from $\mu$ and $u$ be the last cell of $\hat{\mu}$.
(Note that we use the reading described before Definition.\ref{totalorder}.)
For example, $\widehat{(3\,2)}=(3\,1)$ and $u=(2,1)$.
Let $\widehat{\boldsymbol{T}} \in \mathrm{SSTab}(\boldsymbol{\hat{\mu}})$ 
be the $n$-tuple of Young tableaux by removing the last cell from the parts of $\boldsymbol{T}$.
Let $\boldsymbol{T}^{(u)}$, $\dot{\boldsymbol{T}}^{(u)}$ 
 be the $\boldsymbol{T}^{(u)}$ in the definition of $\mathcal{F}$ and the $\dot{\boldsymbol{T}}^{(u)}$ in the definition of $\mathcal{G}$, 
 respectively.

By induction on the size of $\mu$.
If $|\mu|=1$, $\phi_i=s_i$ for all $1 \leq i \leq n-1$.
Thus $\mathcal{F}=\mathcal{G}=$id, therefore (\ref{eq.1}) means (\ref{eq.2}).

Let $|\mu|>1$.
Let
\begin{align}
\boldsymbol{T}' 
  &= \phi_{i_r} \cdots \phi_{i_{k+1}} \boldsymbol{T}^{(u)} \hspace{1em} \text{in Def.$\mathcal{F}$}  
\label{eq.3} \\
  &= \phi_{i_r} \cdots \phi_{i_{k+1}} \phi_{i_k} \cdots \phi_{i_1} \boldsymbol{T} \hspace{1em} \text{in Def.$\mathcal{F}$}.
\label{eq.4}
\end{align}
Then 
\begin{equation*}
  \widehat{\boldsymbol{T}^{(u)}} = \phi_{i_k} \cdots \phi_{i_1} \widehat{\boldsymbol{T}} \hspace{1em} \text{in Def.$\mathcal{F}$}.
\end{equation*}
Hence by induction hypothesis 
\begin{equation*}
  \widehat{\boldsymbol{T}'}=\widehat{\boldsymbol{T}^{(u)}} = s_{i_k} \cdots s_{i_1} \mathcal{F}\widehat{\boldsymbol{T}} 
  \hspace{1em} \text{in Def.$\mathcal{G}$}.
\end{equation*}
Since $\mathcal{F}\widehat{\boldsymbol{T}}=\widehat{\mathcal{F}\boldsymbol{T}}$,
\begin{equation}
  \dot{(\mathcal{F}\boldsymbol{T})}^{(u)}= s_{i_k} \cdots s_{i_1} (\mathcal{F}\boldsymbol{T}) \hspace{1em} \text{in Def.$\mathcal{G}$}.
\label{eq.5}
\end{equation}
Therefore 
\begin{align*}
\dot{(\mathcal{F}\boldsymbol{T})}^{(u)} 
 &= s_{i_k} \cdots s_{i_1} s_{i_1} \cdots s_{i_k} s_{i_{k+1}} \cdots s_{i_r} \boldsymbol{T}' \text{, (by (\ref{eq.4}))} \\
 &= s_{i_{k+1}} \cdots s_{i_r} \boldsymbol{T}' \\
 &= \phi_{i_{k+1}} \cdots \phi_{i_r} \boldsymbol{T}' \text{, (the definition of $\mathcal{F}$)} \\
 &= T^{(u)} \text{, (by (\ref{eq.4}) and the definition of $\mathcal{F}$)}.
\end{align*}

On the other hand, from (\ref{eq.3})
\begin{equation*}
  \boldsymbol{T}'= s_{i_r} \cdots s_{i_{k+1}} \boldsymbol{T}^{(u)} \hspace{1em} \text{in Def.$\mathcal{G}$}.
\end{equation*}
Thus 
\begin{align*}
 \boldsymbol{T}'
 &= s_{i_r} \cdots s_{i_{k+1}} (\dot{\mathcal{F}\boldsymbol{T}})^{(u)} \hspace{1em} \text{in Def.$\mathcal{G}$} \\
 &= s_{i_r} \cdots s_{i_{k+1}} s_{i_k} \cdots s_{i_1} (\mathcal{F}\boldsymbol{T}) \hspace{1em} \text{in Def.$\mathcal{G}$} 
 \text{, (by (\ref{eq.5}))}.
\end{align*}
\end{proof}

\begin{proof}[Proof of Proposition.\ref{keylem}.]

Let $\boldsymbol{T} \in \mathrm{SSTab}(\boldsymbol{\mu})$ and
\begin{equation}
\boldsymbol{T}'=(T_0' \leq \cdots \leq T_{n-1}') = \phi_{i_r} \cdots \phi_{i_1} \boldsymbol{T} \hspace{1em} \text{in Def.$\mathcal{F}$}.
\label{eq.6}
\end{equation}

Then from Lem.\ref{lem-bijection}, 
\begin{equation}
\boldsymbol{T}'= s_{i_r} \cdots s_{i_1} \mathcal{F}\boldsymbol{T} \hspace{1em} \text{in Def.$\mathcal{G}$}.
\label{eq.7}
\end{equation}
Thus
\begin{align*}
\mathcal{G}\mathcal{F}\boldsymbol{T} 
 &= \phi_{i_1} \cdots \phi_{i_r} \boldsymbol{T}' \text{, ((\ref{eq.7}) and the definition of $\mathcal{G}$)} \\
 &= \phi_{i_1} \cdots \phi_{i_r} \phi_{i_r} \cdots \phi_{i_1} \boldsymbol{T} \text{, (by (\ref{eq.6}))} \\
 &= \boldsymbol{T}.
\end{align*}
Since $\mathrm{SSTab}(\boldsymbol{\mu})$ has at most finite elements, $\mathcal{F}$ is bijective.

Next we prove the identity.
From the definition $\mathcal{F}$ and $\mathcal{G}$,
  each application of $\phi_i$ subtracts $1$ from Inv$(\boldsymbol{T})$, and each $s_i$ add $1$ to inv$(\boldsymbol{T}')$.
Hence by (\ref{eq.6}) and (\ref{eq.7}),
$$ \mathrm{Inv}(\boldsymbol{T}) = \mathrm{Inv}(\boldsymbol{T}') + r
 = \mathrm{Inv}(\boldsymbol{T}') + \mathrm{inv}(\mathcal{F}\boldsymbol{T}).$$
\end{proof}

%
%

\section{Decomposition of LLT coefficients}

\subsection{$q$-multinomials}

We prove the following proposition in this subsection.

\begin{proposition}
Let $\nu$ be a composition of $n$, and $k_1,\cdots,k_m \in \mathbb{Z}$.
Then there exists a statistics $h_{k_1,\ldots,k_m} \colon \mathrm{Word}(\nu) \rightarrow \mathbb{Z}$ satisfying

\begin{equation*}
q^{k_1 \nu_1 + \cdots + k_m \nu_m} \biggl[
  \begin{array}{c} n \\ \nu \end{array} \biggr]_q
  = \sum_{w \in \mathrm{Word}(\nu)} q^{n h_{k_1,\ldots,k_m}(w)+\mathrm{inv}(w)}.
\end{equation*}
\label{q-bin2}
\end{proposition}

We shall use the following corollary in next subsection.
It immediately follows from the above proposition and Theorem.\ref{Foata}.

\begin{corollary}
Let $\nu$ be a composition of $n$, and $k_1,\cdots,k_m \in \mathbb{Z}$.
Then there exists a statistics $\alpha'_{k_1,\ldots,k_m} \colon \mathrm{Word}(\nu) \rightarrow \mathbb{Z}$ satisfying

\begin{equation*}
q^{k_1 \nu_1 + \cdots + k_m \nu_m} \biggl[
  \begin{array}{c} n \\ \nu \end{array} \biggr]_q
  = \sum_{w \in \mathrm{Word}(\nu)} q^{n \alpha'_{k_1,\ldots,k_m}(w)+\mathrm{maj}(w)}.
\end{equation*}
\label{q-bin}
\end{corollary}

The statistics $h_{k_1,\ldots,k_m}$ are inductively defined as follows.

We begin with the case of $\nu=(n-a,a)$.

For $0 \leq a \leq n$, $k \in \mathbb{Z}$ and $w \in \mathrm{Word}(n-a,a)$,
  we define $h_{0,k}(w)$ as follows;

\begin{equation*}
  h_{0,k}(w) \coloneqq 
  \begin{cases}  \# \{ j \,|\, w_j=2, \, n-k+1 \leq j \leq n \} & \text{if $k>0$}
              \\ 0                                              & \text{if $k=0$}
              \\ \# \{ j \,|\, w_j=2, \, 1 \leq j \leq -k \}    & \text{if $k<0$} ,   \end{cases} 
\end{equation*}

where if $k>n$ (or $k<-n$), we define $w_{n+1}=w_1,w_{n+2}=w_2,\cdots$ and so on.
In other words $h_{0,k}$ counts''$2$'' in the last (or first) $k$ letters.

Next, we define $h_{k_1,k_2}(w)$ as follows;

\begin{equation*}
h_{k_1,k_2}(w) \coloneqq k_1 + h_{0,-k_1+k_2}(w).
\end{equation*}

Finally, we define the case of $\nu=(\nu_1,\cdots,\nu_m)$.

For $w \in \mathrm{Word}(\nu)$,
  let $w' \in \mathrm{Word}(\nu_2,\ldots,\nu_m)$ be the word obtained by omitting the letter ``$1$'',
  and $w'' \in \mathrm{Word}(\nu_1,n-\nu_1)$ be the word obtained by replacing the letters ``$2$''$\cdots$``$m$'' \
by ``$2$''.
(For example, if $w=12312$ then $w'=232$ and $w''=12212$.)
It is clear that $\mathrm{inv}(w)=\mathrm{inv}(w')+\mathrm{inv}(w'')$ and
  the correspondence $w \leftrightsquigarrow (w',w'')$ is bijective.
Then, we define $h_{k_1,\cdots,k_m}(w)$ as follows;

\begin{equation*}
h_{k_1,\ldots,k_m}(w) \coloneqq h_{k_1,(n-\nu_1)h_{k_2,\ldots,k_m}(w')}(w'') .
\end{equation*}

\vspace{1em}

$\bf{Remark.}$ The notation $w'$ and $w''$ is used only in this section.

\vspace{1em}

Now, we prove Proposition.\ref{q-bin2}.

We start with the case of $\nu=(n-a,a)$.

\begin{lemma}
Let $0 \leq a \leq n$ and $k \in \mathbb{Z}$.
Then

\begin{equation*}
  q^{ka} \biggl[
  \begin{array}{c} n \\ a \end{array} \biggr]_q
  = \sum_{w \in \mathrm{Word}(n-a,a)} q^{n h_{0,k}(w)+\mathrm{inv}(w)}.
\end{equation*}
\end{lemma}

\begin{proof}
For $w=w_1\cdots w_n \in \mathrm{Word}(n-a,a)$, we define $\gamma w$ as

\begin{equation*}
 \gamma w = w_n w_1 \cdots w_{n-1}.
\end{equation*}

Then it is clear that

\begin{equation*}
   \mathrm{inv}(\gamma w)=
   \begin{cases} \mathrm{inv}(w)-a & \text{if $w_n=1$}
              \\ \mathrm{inv}(w)+n-a & \text{if $w_n=2$} \end{cases} .
\end{equation*}

Hence we obtain $\mathrm{inv}(\gamma w)+a=\mathrm{inv}(w)+n h_{0,1}(w)$.

If $k>0$, then
\begin{align*}
\mathrm{inv}(\gamma^k w)+ka &= \mathrm{inv}(\gamma^{k-1} w)+n h_{0,1}(\gamma^{k-1} w)+(k-1)a \\
   & \vdots \\ &= \mathrm{inv}(w)+n ( \sum_{j=1}^k h_{0,1}(\gamma^{k-j} w) )\\ &= \mathrm{inv}(w) + n h_{0,k}(w).
\end{align*}
(Note that if $k>0$, then $h_{0,k}= \sum_{j=1}^k h_{0,1} \circ \gamma^{k-j}$.)

Hence
\begin{align*}
q^{ka} \biggl[ \begin{array}{c} n \\ a \end{array} \biggr]_q
  &= \sum_{w \in \mathrm{Word}(n-a,a)} q^{\mathrm{inv}(w)+ka} \\
  &= \sum_{w \in \mathrm{Word}(n-a,a)} q^{\mathrm{inv}(\gamma^k(w))+ka} \\
  &= \sum_{w \in \mathrm{Word}(n-a,a)} q^{\mathrm{inv}(w)+n h_{0,k}(w)} .
\end{align*}

If $k<0$, we can prove the assertion similarly.
\end{proof}

\vspace{2em}

\begin{corollary}
Let $0 \leq a \leq n$ and $k_1,k_2 \in \mathbb{Z}$.
Then there exists a statistics $h_{k_1,k_2} \colon \mathrm{Word}(n-a,a) \rightarrow \mathbb{Z}$ satisfying

\begin{equation*}
  q^{k_1(n-a)+k_2a} \biggl[
  \begin{array}{c} n \\ a \end{array} \biggr]_q
  = \sum_{w \in \mathrm{Word}(n-a,a)} q^{n h_{k_1,k_2}(w)+\mathrm{inv}(w)}.
\end{equation*}

\label{q-bincor}
\end{corollary}

\begin{proof}
Since $k_1(n-a)+k_2a=k_1n+(-k_1+k_2)a$, set $h_{k_1,k_2}(w)=k_1+h_{0,-k_1+k_2}(w)$.
\end{proof}

\vspace{2em}

\begin{proof}[Proof of Proposition.\ref{q-bin2}.]
We may assume that the length of $\nu$ is equal to $m$.
We prove the assertion by induction on $m$.
If $m=1$, it is clear.
Suppose $m>1$.
Then

\begin{align*}q^{k_1 \nu_1 + \cdots + k_m \nu_m} \biggl[
  \begin{array}{c} n \\ \nu \end{array} \biggr]_q
  = q^{k_1 \nu_1} \biggl[  \begin{array}{c} n \\ \nu_1 \end{array} \biggr]_q
    q^{k_2 \nu_2 + \cdots + k_m \nu_m} \biggl[
    \begin{array}{c} n-\nu_1 \\ \nu_2,\ldots,\nu_m \end{array} \biggr]_q .
\tag{*}
\end{align*}

By induction hypothesis, there exists a statistics
  $h_{k_2,\ldots,k_m} \colon \mathrm{Word}(\nu_2,\ldots,\nu_m) \rightarrow \mathbb{Z}$ satisfying

\begin{equation*}
  q^{k_2 \nu_2 + \cdots + k_m \nu_m} \biggl[
  \begin{array}{c} n \\ \nu \end{array} \biggr]_q
  = \sum_{w' \in \mathrm{Word}(\nu_2,\ldots,\nu_m)} q^{(n-\nu_1) h_{k_2,\ldots,k_m}(w')+\mathrm{inv}(w')}.
\end{equation*}

Hence

\begin{align*}
(*) &= q^{k_1 \nu_1} \biggl[  \begin{array}{c} n \\ \nu_1 \end{array} \biggr]_q
       \sum_{w' \in \mathrm{Word}(\nu_2,\ldots,\nu_m)} q^{(n-\nu_1) h_{k_2,\ldots,k_m}(w')+\mathrm{inv}(w')} \\
    &= \sum_{w' \in \mathrm{Word}(\nu_2,\ldots,\nu_m)} q^{k_1\nu_1+(n-\nu_1) h_{k_2,\ldots,k_m}(w')}
       \biggl[ \begin{array}{c} n \\ \nu_1 \end{array} \biggr]_q q^{\mathrm{inv}(w')}  \\
    &= \sum_{w' \in \mathrm{Word}(\nu_2,\ldots,\nu_m)}\sum_{w'' \in \mathrm{Word}(\nu_1,n-\nu_1)}
       q^{n h_{k_1,\ldots,k_m}(w)+\mathrm{inv}(w'')+\mathrm{inv}(w')} , \text{(By Cor.\ref{q-bincor})}  \\
    &= \sum_{w \in \mathrm{Word}(\nu)} q^{n h_{k_1,\ldots,k_m}(w)+\mathrm{inv}(w)},
\end{align*}
where $h_{k_1,\ldots,k_m}(w)=h_{k_1,(n-\nu_1)h_{k_2,\ldots,k_m}(w')}(w'')$ for $w \in \mathrm{Word}(\nu)$.
\end{proof}


\subsection{Decomposition of LLT coefficients}

From now on, we consider the case where $\mu^{(0)}=\mu^{(1)} = \cdots = \mu^{(n-1)} = \mu $, i.e.
$$ \boldsymbol{\mu} = \underbrace{(\mu,\cdots,\mu)}_{n} . $$

\begin{theorem}
Let $ \boldsymbol{\mu} = \underbrace{(\mu,\cdots,\mu)}_{n} .$
There is a statistics $\alpha \colon \mathrm{SSTab}(\boldsymbol{\mu}) 
     \rightarrow \mathbb{Z}$ satisfying
$$ G_{\boldsymbol{\mu},\nu}(q) = \sum_{\boldsymbol{T} \in \mathrm{SSTab}(\boldsymbol{\mu},\nu)} 
                                     q^{n \alpha(\boldsymbol{T}) + \mathrm{maj}(\boldsymbol{T}) + d_{\boldsymbol{\mu}}} , $$
  where $ d_{\boldsymbol{\mu}} = \mathrm{min} \{ \mathrm{Inv}(\boldsymbol{T})| \boldsymbol{T} \in \mathrm{SSTab}(\boldsymbol{\mu}) \} $ 
  is the minimum Inversion number on the set of semistandard tableaux of shape $\boldsymbol{\mu}$.
\label{mainthm}
\end{theorem}

we use some facts to prove this theorem.

\begin{lemma}
Let $T_0 \leq \cdots \leq T_{n-1} \in \mathrm{SSTab}(\mu)$
   and $\rho$ be the composition determined by
   $T_0=\cdots=T_{\rho_1 -1} < T_{\rho_1}=\cdots=T_{\rho_1+\rho_2-1} < T_{\rho_1+\rho_2}= \cdots.$
Then there are $k_1,k_2,\cdots \in \mathbb{Z}$ satisfying
$$ \mathrm{Inv}(T_0,\cdots,T_{n-1}) = d_{\boldsymbol{\mu}} + k_1\rho_1 + k_2\rho_2 + \cdots ,$$
where $d_{\boldsymbol{\mu}}$ is the minimum Inversion number $($see $\mathrm{Theorem}.\ref{mainthm})$.
\label{lemma11}
\end{lemma}

\begin{proof}
First we count the inversion numbers of $U_k=(\underbrace{T_0,\cdots,T_0}_{k})$ for some $T_0 \in \mathrm{SSTab}(\mu)$.
Then $U_k$ has $\binom{k}{2} d_1(\mu)$ inversions of type(i) in the definition of the Inversion numbers
  and $\binom{k}{2} d_2(\mu)$ Inversions of type(ii),
  where $$d_1(\mu)=\sum_{(i.j) \in \mu} \mathrm{min} (i,j) \, \text{ and } \, d_2(\mu)=\sum_{(i.j) \in \mu} \mathrm{min} (i-1,j). $$
Note that it is independent of the choice of $T_0$.

On the other hand, it is known that if $\boldsymbol{T}$ has the minimal Inversion number,
  then the ribbon tableau corresponding to $\boldsymbol{T}$ has the maximal spin(see \cite[Lemma.5.2.2.]{HH1}).
Moreover the ribbon tableau corresponding to $U_n$ has the maximal spin.
Therefore we have
$$d_{\boldsymbol{\mu}}=\binom{n}{2} ( d_1(\mu) + d_2(\mu) ) .$$

Finally, there are non-negative integers $b_1,b_2,\cdots \in \mathbb{Z}_{\geq 0}$ satisfying
\begin{align*}
  \mathrm{Inv}(T_0,\cdots,T_{n-1})
   &= \sum_i b_i \rho_i + \sum_i \binom{\rho_i}{2} ( d_1(\mu) + d_2(\mu) ) \\
   &= \sum_i b_i \rho_i + ( \sum_i \binom{\rho_i}{2} - \binom{n}{2} ) ( d_1(\mu) + d_2(\mu) ) + d_{\boldsymbol{\mu}} .
\end{align*}

But, there are integers $b_1',b_2',\cdots \in \mathbb{Z}$ satisfying
$$ \sum_i \binom{\rho_i}{2} - \binom{n}{2} = \sum_i b_i' \rho_i . $$
\end{proof}

\vspace{2em}

\begin{proof}[Proof of Theorem.\ref{mainthm}.]
From Corollary.\ref{fermionicformula}, we have
$$ G_{\boldsymbol{\mu},\nu}(q) = \sum_{T_0 \leq \cdots \leq T_{n-1}} q^{\mathrm{Inv}(T_0, \cdots ,T_{n-1})}
  \biggl[ \begin{array}{c} n \\ \rho \end{array} \biggr]_q ,$$
where $\rho$ be the composition determined from $T_0 \leq \cdots \leq T_{n-1}$.
By Lemma.\ref{lemma11}, there are positive integers $k_1,k_2,\cdots$ satisfying
$ \mathrm{Inv}(T_0,\cdots,T_{n-1}) = d_{\boldsymbol{\mu}} + k_1\rho_1 + k_2\rho_2 + \cdots .$
Thus by Corollary.\ref{q-bin}, there is a statistics $\alpha'_{k_1,k_2,\ldots} \colon \mathrm{Word}(\rho) \rightarrow \mathbb{Z}$ satisfying
\begin{align*}
G_{\boldsymbol{\mu},\nu}(q)
  &= \sum_{T_0 \leq \cdots \leq T_{n-1}} \sum_{w \in \mathrm{Word}(\rho)} 
     q^{n \alpha'_{k_1,k_2,\ldots}(w) + \mathrm{maj}(w) + d_{\boldsymbol{\mu}} } .
\end{align*}

For each $\boldsymbol{T} \in \mathrm{SSTab}(\boldsymbol{\mu},\nu)$, 
we define $\alpha(\boldsymbol{T})$ is defined as follows;

Let $T_0' \leq \cdots \leq T_{n-1}'$ be the rearrangement of the components of $\boldsymbol{T}$.
Let $\rho$ be the composition determined from $T_0' \leq \cdots \leq T_{n-1}'$.
Let $k_1,k_2,\cdots$ be the integers satisfying
$ \mathrm{Inv}(T_0',\cdots,T_{n-1}') = d_{\boldsymbol{\mu}} + k_1\rho_1 + k_2\rho_2 + \cdots .$
Let $w \in \mathrm{Word}(\rho)$ be the word determined from $\boldsymbol{T}$.

Then $\mathrm{maj}(w) = \mathrm{maj} (\boldsymbol{T})$, and we define 
\begin{equation*}
\alpha(\boldsymbol{T}) \coloneqq \alpha'_{k_1,k_2,\ldots}(w).
\end{equation*}

Then

\begin{align*}
G_{\boldsymbol{\mu},\nu}(q)
  &= \sum_{\boldsymbol{T} \in \mathrm{SSTab}(\boldsymbol{\mu},\nu)}
     q^{n \alpha(\boldsymbol{T}) + \mathrm{maj}(\boldsymbol{T}) + d_{\boldsymbol{\mu}} } .
\end{align*}

\end{proof}

\vspace{1em}

Let $G_{\boldsymbol{\mu},\nu}(q)=\sum_{j \geq 0} a_j q^j$ be the LLT coefficient.
For $0 \leq i \leq n-1$, set 
$$ G_{\boldsymbol{\mu},\nu}^{(i)}(q) \coloneqq  \sum_{j \geq 0} a_{jn+i+d_{\boldsymbol{\mu}}} q^{jn+i+d_{\boldsymbol{\mu}}} .$$ 
From the definition it is clear that $G_{\boldsymbol{\mu},\nu}(q) = \sum_{i=0}^{n-1} G_{\boldsymbol{\mu},\nu}^{(i)}(q) $.
Moreover from Theorem.\ref{mainthm}, 
$$ G_{\boldsymbol{\mu},\nu}^{(i)}(q) 
   = \sum_{\substack{\boldsymbol{T} \in \mathrm{SSTab}(\boldsymbol{\mu},\nu) \\ \mathrm{maj}(\boldsymbol{T}) \equiv i \, \mathrm{mod} \,n}}
     q^{n \alpha(\boldsymbol{T}) + i + d_{\boldsymbol{\mu}}} . $$

\begin{example}
\rm{Let} $\boldsymbol{\mu} = ((2),(2),(2))$ and $\nu=(4\,2)$.
Then $d_{\boldsymbol{\mu}}=0$ and $G_{\boldsymbol{\mu},\nu}(q)=1+q+2q^2+q^3+q^4$.
Thus, 
$$ G_{\boldsymbol{\mu},\nu}^{(0)}(q)=1+q^3 ,\, G_{\boldsymbol{\mu},\nu}^{(2)}(q)=q+q^4 ,\, 
   G_{\boldsymbol{\mu},\nu}^{(2)}(q)=2q^2 .$$ 
\end{example}

%
%

\section{$\widetilde{LR}_{\boldsymbol{\mu},\nu}^{(i)}(q)$ and representation theory}

Let $\mathfrak{S}_n$ be the $n$-th symmetric group, and fix an $n$-cycle $\gamma \in \mathfrak{S}_n$.
$($for example $\gamma=(1\,2\,\cdots\,n))$.
Let $\zeta_n \in \mathbb{C}$ be a primitive $n$-th root of unity.
For an $\mathfrak{S}_n$-module $V$,
  we denote $V[\zeta_n^i]$ by the $\zeta_n^i$-eigenspace of $\gamma$ $(0 \leq i \leq n-1)$.

Since $\mathfrak{S}_n$ acts on $V_\mu^{\otimes n}$ by permutating the components,
 $V_\mu^{\otimes n}[\zeta_n^i]$ is a $GL_N$-submodule of $V_\mu^{\otimes n}$.
We discuss a $q$-analog of the multiplicities in $V_\mu^{\otimes n}[\zeta_n^i]$.
More precisely;
\par $({\mathrm{i}}).$ we introduce the combinatorial object corresponding to
            $\mathrm{SSTab}(\boldsymbol{\mu})$.
\par $({\mathrm{ii}}).$ we define a $q$-analog of the multiplicities of $V_\mu^{\otimes n}[\zeta_n^i]$
              like Definition.\ref{q-LR}, and prove it belong to $\mathbb{Z}_{\geq 0}[q]$.

\vspace{1em}

For $\nu \vdash n |\mu|$ and $0 \leq i \leq n-1$, set
$$ \mathrm{SSTab}(\boldsymbol{\mu},\nu;i) = \{ \boldsymbol{T} \in \mathrm{SSTab}(\boldsymbol{\mu},\nu) |
          \mathrm{maj}(\boldsymbol{T}) \equiv i\,\, \mathrm{mod} \,\,n  \}, $$
and $K_{\boldsymbol{\mu},\nu}^{(i)} \coloneqq  \# \mathrm{SSTab}(\boldsymbol{\mu},\nu;i)$.

\begin{proposition}
Let $\nu \vdash n |\mu|$.
Then
$$ \mathrm{dim}(V_\mu^{\otimes n})[\zeta_n^i](\nu) = K_{\boldsymbol{\mu},\nu}^{(i)}, $$
where for a $GL_N$-module $V$, $V(\lambda)$ denotes the $\nu$-weight space of $V$.
\label{append2}
\end{proposition}

We postpone the proof for section 7.

\vspace{1em}

By Proposition.\ref{append2} and Theorem.\ref{mainthm},
   $G_{\boldsymbol{\mu},\nu}^{(i)}(q)$ is a $q$-analog of $\mathrm{dim}(V_\mu^{\otimes n})[\zeta_n^i](\nu)$.
So, $\widetilde{LR}_{\boldsymbol{\mu},\nu}^{(i)}(q) \coloneqq
       \sum_{\rho \vdash n}K_{\rho,\nu}^{-1}G_{\boldsymbol{\mu},\rho}^{(i)}(q)$
   is a $q$-analog of $[(V_\mu^{\otimes n})[\zeta_n^i] \colon V_\nu]$.

The next Corollary is clear from Theorem.\ref{positivity} because each coefficient which differs from $0$ in
  $\widetilde{LR}_{\boldsymbol{\mu},\nu}^{(i)}(q)$ is equal to the coefficient of same degree in
    $\widetilde{LR}_{\mu^{(0)},\cdots,\mu^{(n-1)}}^{\nu}(q)$.

\begin{corollary}
For $0 \leq i \leq n-1$,
$$ \widetilde{LR}_{\boldsymbol{\mu},\nu}^{(i)}(q) \in \mathbb{Z}_{\geq 0}[q] . $$
\end{corollary}

%
%

\section{$q$-analog of a sum of the plethysm multiplicities}

\subsection{standard tableau and its major index}

We review the standard tableaux.

Let $\lambda$ be a partition of $n$.
As usual, a \textit{standard tableau} is a Young tableau which contains each number $1,2,\cdots,n$ exactly once.
Let STab$(\lambda)$ denotes the set of standard tableaux of shape $\lambda$, 
   and set $\mathrm{STab}(n) = \bigsqcup_{\lambda \vdash n} \mathrm{STab}(\lambda)$ be the set of standard tableaux whose size are equal to $n$.

\begin{definition}
Let $S \in \mathrm{STab}(n)$ be a standard tableau.
A \textit{decent} of $S$ is an integer $i$, $1 \leq i \leq n-1$, 
  for which the row of the cell filled by $i+1$ is strictly below that of $i$.
A major index of $S$ is defined to be the sum of its descents.
\end{definition}

The relation between the major index of a word defined in Definition.\ref{major} and the major index defined the above is as follows.
Let $w \in \Gamma^n$.
Applying the Robinson-Schensted algorithm to $w$ with respect to the total order on $\Gamma$,  
we get $(P,Q) \in \mathrm{SSTab}(\lambda)\times \mathrm{STab}(\lambda)$ for some $\lambda \vdash n$.
We write $w \leftrightarrow (P,Q)$.
Then it is easy to check $\mathrm{maj}(w) = \mathrm{maj}(Q) ($cf.\cite{BS}$)$.

\begin{example}
\rm{Let} $w=12142$.
By applying Robinson-Schensted correspondence, we get $w \leftrightarrow (P,Q) = 
\unitlength 0.1in
\begin{picture}( 18 , 3.5  )(  2.45 , -4.5 )
%
\special{pn 8}%
\special{pa 400 600}%
\special{pa 800 600}%
\special{pa 800 200}%
\special{pa 400 200}%
\special{pa 400 600}%
\special{fp}%
%
\special{pn 8}%
\special{pa 400 400}%
\special{pa 1000 400}%
\special{pa 1000 200}%
\special{pa 400 200}%
\special{pa 400 400}%
\special{fp}%
%
\special{pn 8}%
\special{pa 600 600}%
\special{pa 600 200}%
\special{fp}%
\put(5.0000,-5.0000){\makebox(0,0){2}}%
\put(7.0000,-5.0000){\makebox(0,0){4}}%
\put(5.0000,-3.0000){\makebox(0,0){1}}%
\put(7.0000,-3.0000){\makebox(0,0){1}}%
\put(9.0000,-3.0000){\makebox(0,0){2}}%
%
\special{pn 8}%
\special{pa 1200 600}%
\special{pa 1600 600}%
\special{pa 1600 200}%
\special{pa 1200 200}%
\special{pa 1200 600}%
\special{fp}%
%
\special{pn 8}%
\special{pa 1200 400}%
\special{pa 1800 400}%
\special{pa 1800 200}%
\special{pa 1200 200}%
\special{pa 1200 400}%
\special{fp}%
%
\special{pn 8}%
\special{pa 1400 600}%
\special{pa 1400 200}%
\special{fp}%
\put(13.0000,-3.0000){\makebox(0,0){1}}%
\put(2.9000,-4.0000){\makebox(0,0){(}}%
\put(10.9000,-4.0000){\makebox(0,0){,}}%
\put(18.9000,-4.0000){\makebox(0,0){)}}%
\put(15.0000,-3.0000){\makebox(0,0){2}}%
\put(17.0000,-3.0000){\makebox(0,0){4}}%
\put(13.0000,-5.0000){\makebox(0,0){3}}%
\put(15.0000,-5.0000){\makebox(0,0){5}}%
\end{picture}%
$.

\vspace{1em}

Then maj$(w)=2+4=6$ and maj$(Q)=2+4=6$.
\end{example}


\subsection{plethysm}

Next we review the plethysm for $GL_N$-modules.

For a given standard tableau $S \in \mathrm{STab}(n)$, 
 $Y_S \in \mathbb{C} \mathfrak{S}_n$ denotes the Young symmetrizer corresponding to $S$.
Let $\mu$ be a Young diagram, and $V_\mu$ be the irreducible $GL_N$-module corresponding to $\mu$.

\begin{definition}
We define a $GL_N$-module $V_{S[\mu]}$ as follows;
$$ V_{S[\mu]} \coloneqq Y_S (V_{\mu}^{\otimes n}) ,$$
  where each element in $\mathfrak{S}_n$ acts on $V_{\mu}^{\otimes n}$ by permutating the components.
\end{definition}

\begin{example}
If $\lambda = (n)$, and $S$ is the standard tableau of shape $(n)$,
then $Y_S=\sum_{\sigma \in \mathfrak{S}_n} \sigma$ is the symmetrizer
  and $ V_{S[\mu]} = \boldsymbol{S}^n (V_\mu)$ is the n-th symmetric product of $V_\mu$. 

Similarly, in the case of $\lambda = (1^n)$,
  $V_{S[\mu]} = \Lambda^n (V_\mu)$ is the n-th exterior product of $V_\mu$.
\end{example}

$\bf{Remark}$.
It is well-known that the $GL_N$-module structure on $V_{S[\mu]}$ depends only on the shape of $S$.
Namely if $S, S' \in \mathrm{STab}(\lambda)$ for some $\lambda \vdash n$, then $V_{S[\mu]} \simeq V_{S'[\mu]}$ as $GL_N$-module.
However we do NOT identify $V_{S[\mu]}$ and $V_{S'[\mu]}$.
Because the $q$-analog, given in the next section, of the multiplicities $V_{S[\mu]}$ 
  do NOT agree with that of $V_{S'[\mu]}$.

\begin{definition}
If 
  $$ V_{S[\mu]} \simeq \bigoplus_{\nu \vdash n |\mu|} V_{\nu}^{\oplus a_{S[\mu]}^{\nu}} $$
  as $GL_N$-module, the numbers $a_{S[\mu]}^{\nu}$ are called plethysm multiplicities.
\end{definition}

Next proposition is also well-known fact$($cf.??$)$.

\begin{proposition}
$$ V_{\mu}^{\otimes n} \simeq \bigoplus_{S \in \mathrm{STab}(n)} V_{S[\mu]}. $$
\end{proposition}


\subsection{$q$-analog of a sum of plethysm multiplicities}

We discuss a $q$-analog of the plethysm multiplicities.

We note the following result.

\begin{proposition}
For $0 \leq i \leq n-1$, 
$$ (V_{\mu}^{\otimes n})[\zeta_n^i] \simeq 
  \bigoplus_{S \in \mathrm{STab}(n),\, \mathrm{maj}(S) \equiv i \, \mathrm{mod}\,\,n} V_{S[\mu]}$$
as a $GL_N$-module.
\label{append1}
\end{proposition}

\begin{proof}
From the Schur-Weyl duality, we have
$$ V_\mu^{\otimes n} \simeq \bigoplus_{\lambda \vdash n} S^{\lambda} \boxtimes V_{\lambda [\mu]} $$
as a $\mathfrak{S}_n \times GL_N$-bimodule.
Then by Corollary.\ref{majorcor}, we have
\begin{align*}
 V_\mu^{\otimes n}[\zeta_n^i]
  & \simeq \bigoplus_{\lambda \vdash n} V_{\lambda [\mu]}^{\oplus K_{\lambda}^{(i)}} \\
  & \simeq \bigoplus_{\substack{S \in \mathrm{STab}(n) \\ \mathrm{maj}(S) \equiv i \text{ mod } n}} V_{S[\mu]}
\end{align*}
as a $GL_N$-module.
\end{proof}

\vspace{1em}

Therefore we can interpret $\widetilde{LR}_{\boldsymbol{\mu},\nu}^{(i)}(q)$ as a $q$-analog of a sum of plethysm multiplicities 
  since $\widetilde{LR}_{\boldsymbol{\mu},\nu}^{(i)}(q)$ is a $q$-analog of the multiplicity $[(V_\mu^{\otimes n})[\zeta_n^i] \colon V_\nu]$.

\begin{corollary}
For $0 \leq i \leq n-1$,
$$ \widetilde{LR}_{\boldsymbol{\mu},\nu}^{(i)}(1) = 
   \sum_{\substack{S \in \mathrm{STab}(n) \\ \mathrm{maj}(S) \equiv i \, \mathrm{mod}\,\,n}} a_{S[\mu]}^{\nu} 
   = \sum_{\lambda \vdash n} a_{\lambda[\mu]}^{\nu} K_{\lambda}^{(i)} ,$$
where $K_{\lambda}^{(i)}=\#  \{ S \in \mathrm{STab}(\lambda) \,|\, \mathrm{maj}(S) \equiv i \,\, \mathrm{mod} \,\, n \}$. 
\end{corollary}


\subsection{some problems about a $q$-analog of the plethysm multiplicities}

From the result in the previous section, 
  we must separate the polynomial $G_{\boldsymbol{\mu},\nu}^{(i)}(q)$ to get a $q$-analog of the plethysm multiplicities.
Firstly we separate $\mathrm{SSTab}(\boldsymbol{\mu},\nu)$ and determine the combinatorial object.

\begin{definition}
For a Young diagram $\mu$ and $S \in \mathrm{STab}(n)$, 
$$ \mathrm{SSTab}(S[\boldsymbol{\mu}],\nu) \coloneqq 
   \{ \boldsymbol{T} \in \mathrm{SSTab}(\boldsymbol{\mu},\nu) | \, \boldsymbol{T} \leftrightarrow (\,\cdot\,,S) \} , $$
    where $\boldsymbol{T} \leftrightarrow (\,\cdot\,,S)$ means that $\boldsymbol{T}$ corresponds to $(\,\cdot\,,S)$ 
    through the Robinson-Schensted correspondence with respect to the total order described in Definition.\ref{totalorder}. 
\end{definition}

$\bf{Remark.}$
If the shape of $S$ is $\lambda$,
then we can regard $\mathrm{SSTab}(S[\boldsymbol{\mu}],\nu)$ as the set of Young tableau of shape $\lambda$ 
  whose entries are semi-standard tableaux of shape $\mu$ and the sum of the weight are equal to $\nu$ (see \cite{I}).

\begin{example}
\rm{Let}  $S= 
\unitlength 0.1in
\begin{picture}(  4.0000,  4.0000)(  2.0000, -3.5000)
%
\special{pn 8}%
\special{pa 200 600}%
\special{pa 400 600}%
\special{pa 400 200}%
\special{pa 200 200}%
\special{pa 200 600}%
\special{fp}%
%
\special{pn 8}%
\special{pa 200 400}%
\special{pa 600 400}%
\special{pa 600 200}%
\special{pa 200 200}%
\special{pa 200 400}%
\special{fp}%
\put(3.0000,-3.0000){\makebox(0,0){1}}%
\put(5.0000,-3.0000){\makebox(0,0){2}}%
\put(3.0000,-5.0000){\makebox(0,0){3}}%
\end{picture}%
$, $\mu=(2)$, and $\nu=(4\,2)$.

\vspace{1em}

Then $\mathrm{SSTab}(\boldsymbol{\mu},\nu)$ has $6$ tableaux, and 
$$ \mathrm{SSTab}(S[\boldsymbol{\mu}],\nu)= \{    
\unitlength 0.1in
\begin{picture}( 18.0000,  2.0000)(  2.5500, -4.0000)
%
\special{pn 8}%
\special{pa 400 400}%
\special{pa 800 400}%
\special{pa 800 200}%
\special{pa 400 200}%
\special{pa 400 400}%
\special{fp}%
%
\special{pn 8}%
\special{pa 600 400}%
\special{pa 600 200}%
\special{fp}%
%
\special{pn 8}%
\special{pa 1000 400}%
\special{pa 1400 400}%
\special{pa 1400 200}%
\special{pa 1000 200}%
\special{pa 1000 400}%
\special{fp}%
%
\special{pn 8}%
\special{pa 1600 400}%
\special{pa 2000 400}%
\special{pa 2000 200}%
\special{pa 1600 200}%
\special{pa 1600 400}%
\special{fp}%
%
\special{pn 8}%
\special{pa 1200 400}%
\special{pa 1200 200}%
\special{fp}%
%
\special{pn 8}%
\special{pa 1800 400}%
\special{pa 1800 200}%
\special{fp}%
\put(3.0000,-3.0000){\makebox(0,0){(}}%
\put(5.0000,-3.0000){\makebox(0,0){1}}%
\put(9.0000,-3.0000){\makebox(0,0){,}}%
\put(11.0000,-3.0000){\makebox(0,0){1}}%
\put(13.0000,-3.0000){\makebox(0,0){2}}%
\put(15.0000,-3.0000){\makebox(0,0){,}}%
\put(17.0000,-3.0000){\makebox(0,0){1}}%
\put(21.0000,-3.0000){\makebox(0,0){)}}%
\put(7.0000,-3.0000){\makebox(0,0){2}}%
\put(19.0000,-3.0000){\makebox(0,0){1}}%
\end{picture}%
 \,\,\,\,,\,\, 
\unitlength 0.1in
\begin{picture}( 18.0000,  2.0000)(  2.5500, -4.0000)
%
\special{pn 8}%
\special{pa 400 400}%
\special{pa 800 400}%
\special{pa 800 200}%
\special{pa 400 200}%
\special{pa 400 400}%
\special{fp}%
%
\special{pn 8}%
\special{pa 600 400}%
\special{pa 600 200}%
\special{fp}%
%
\special{pn 8}%
\special{pa 1000 400}%
\special{pa 1400 400}%
\special{pa 1400 200}%
\special{pa 1000 200}%
\special{pa 1000 400}%
\special{fp}%
%
\special{pn 8}%
\special{pa 1600 400}%
\special{pa 2000 400}%
\special{pa 2000 200}%
\special{pa 1600 200}%
\special{pa 1600 400}%
\special{fp}%
%
\special{pn 8}%
\special{pa 1200 400}%
\special{pa 1200 200}%
\special{fp}%
%
\special{pn 8}%
\special{pa 1800 400}%
\special{pa 1800 200}%
\special{fp}%
\put(3.0000,-3.0000){\makebox(0,0){(}}%
\put(5.0000,-3.0000){\makebox(0,0){1}}%
\put(7.0000,-3.0000){\makebox(0,0){1}}%
\put(9.0000,-3.0000){\makebox(0,0){,}}%
\put(13.0000,-3.0000){\makebox(0,0){2}}%
\put(15.0000,-3.0000){\makebox(0,0){,}}%
\put(17.0000,-3.0000){\makebox(0,0){1}}%
\put(21.0000,-3.0000){\makebox(0,0){)}}%
\put(11.0000,-3.0000){\makebox(0,0){2}}%
\put(19.0000,-3.0000){\makebox(0,0){1}}%
\end{picture}%
 \,\,\, \} .$$
\end{example}

This is the combinatorial object corresponding to $\mathrm{SSTab}(\boldsymbol{\mu})$ in the case of 
  the $q$-analog of Littlewood-Richardon coefficients.

\begin{proposition}[{\cite[Theorem.3.5.]{I}}] 
\par
For a standard tableau $S \in \mathrm{STab}(n)$,
$$ \mathrm{dim}V_{S[\mu]}(\nu) = \# \mathrm{SSTab}(S[\boldsymbol{\mu}],\nu). $$
\label{plethysmthm}
\end{proposition}

Since the Robinson-Schensted correspondence preserves the major index, the following holds.

\begin{lemma}
For $0 \leq i \leq n-1$, 
$$ \mathrm{SSTab}(\boldsymbol{\mu},\nu;i) = \bigsqcup_{S \in \mathrm{STab}(n),\, \mathrm{maj}(S) \equiv i \, \mathrm{mod}\,\,n} 
  \mathrm{SSTab}(S[\boldsymbol{\mu}],\nu).  $$
\end{lemma}

Thus, by restricting $\alpha$, (which appears in Theorem.\ref{mainthm}), to $\mathrm{SSTab}(S[\boldsymbol{\mu}],\nu)$, 
we get a polynomial
$$ G_{S[\mu],\nu}(q) \coloneqq \sum_{\boldsymbol{T} \in \mathrm{SSTab}(S[\boldsymbol{\mu}],\nu)} 
   q^{n \alpha(\boldsymbol{T}) + \mathrm{maj}(S) + d_{\boldsymbol{\mu}}} . $$
Set 
$$ a_{S[\boldsymbol{\mu}]}^{\nu}(q)=
  \sum_{\rho \vdash n}K_{\rho,\nu}^{-1} G_{S[\boldsymbol{\mu}],\rho}(q) \,\,, $$
where $K_{\rho,\nu}^{-1}$ is the inverse Kostka number defining in section $2$.

However NOT all $a_{S[\boldsymbol{\mu}]}^{\nu}(q)$ belong to $\mathbb{Z}_{\geq 0}[q]$,
that is $a_{S[\boldsymbol{\mu}]}^{\nu}(q)$ may has \textit{negative} coefficients.

\begin{problem}
There is a statistics $\widetilde{\alpha} \colon \mathrm{SSTab}(S[\mu]) \rightarrow \mathbb{Z}$ satisfying
$({\mathrm{i}})$. $\widetilde{a}_{S[\boldsymbol{\mu}]}^{\nu}(q) \in \mathbb{Z}_{\geq 0}[q],$ and  

$({\mathrm{ii}})$. $\alpha$ and $\widetilde{\alpha}$ have the same distribution on $\mathrm{SSTab}(\boldsymbol{\mu},\nu;i)$, that is
$$ \sum_{\boldsymbol{T} \in \mathrm{SSTab}(\boldsymbol{\mu},\nu;i)}
   q^{n \widetilde{\alpha}(\boldsymbol{T}) + i + d_{\boldsymbol{\mu}}} = G_{\boldsymbol{\mu},\nu}^{(i)}(q) $$
  for any $i$.

Where $\widetilde{a}_{S[\boldsymbol{\mu}]}^{\nu}(q)$ is the polynomial obtained from $\widetilde{\alpha}$ instead of $\alpha$.
\end{problem}

\subsection{Proof of Proposition.\ref{append2}}

Before we proof Proposition.\ref{append2}, we prepare some results.

Let $S^{\lambda}$ be the irreducible representation of $\mathfrak{S}_n$ corresponding to $\lambda \vdash n$.
Let $C_n$ be the cyclic group generated by $\gamma$,
 and $\zeta_n^i$ denotes the 1-dimensional representation of $C_n$.

\vspace{1em}

We use the following theorem.

\begin{theorem}[\cite{KW}, Corollary of Theorem.1]
$$ \mathrm{Ind}_{C_n}^{\mathfrak{S}_n} (\zeta_n^i) \simeq
  \bigoplus_{\lambda \vdash n} {S^{\lambda}}^{ \oplus K_{\lambda}^{(i)} } $$
as a $\mathfrak{S_n}$-module,
  where $K_{\lambda}^{(i)}=\#  \{ S \in \mathrm{STab}(\lambda) \,|\, \mathrm{maj}(S) \equiv i \,\, \mathrm{mod} \,\, n \}$.
\end{theorem}

Thus from Frobenius reciprocity, we have the followings.

\begin{corollary}
$$ \mathrm{dim} S^{\lambda}[\zeta_n^i] = K_{\lambda}^{(i)}. $$
\label{majorcor}
\end{corollary}

\begin{proof}[Proof of Proposition.\ref{append2}]
From the Schur-Weyl duality,
$$ V_\mu^{\otimes n}[\zeta_n^i](\nu) \simeq \bigoplus_{\lambda \vdash n} S^{\lambda}[\zeta_n^i] \otimes V_{\lambda [\mu]}(\nu) ,$$
as a $\mathbb{C}$-linear space.
Hence by Corollary.\ref{majorcor} and Proposition.\ref{plethysmthm},
\begin{align*}
\mathrm{dim} V_\mu^{\otimes n}[\zeta_n^i](\nu)
  & = \sum_{\lambda \vdash n} K_{\lambda}^{(i)} K_{\lambda[\mu],\nu} \\
  & = K_{\boldsymbol{\mu},\nu}^{(i)},
\end{align*}
where the last equality follows from the Robinson-Schensted correspondence.

\end{proof}

%
%

\end{document}